\documentclass[a4paper,12pt]{amsart}

\usepackage{amssymb,amsmath,amsthm,mathrsfs,enumerate,graphicx, color}
\usepackage[pdfpagelabels,colorlinks,linkcolor=blue,citecolor=blue,urlcolor=blue]{hyperref}
\usepackage{esint}
\usepackage{tikz}
\usetikzlibrary{arrows}
\usepackage{mathtools}
\usepackage{bm}
\usepackage{cite}
\allowdisplaybreaks
\mathtoolsset{showonlyrefs}

\newtheorem{thm}{Theorem}[section]

\newtheorem{lem}[thm]{Lemma}
\newtheorem{prop}[thm]{Proposition}
\newtheorem{defn}[thm]{Definition}
\newtheorem{rem}[thm]{Remark}

\newcommand{\lesi}{\lesssim}

\newcommand{\supp}{\operatorname{supp}}
\newcommand{\f}{\frac}

\newcommand{\vc}{\infty}

\newcommand{\N}{\mathfrak{N}}

\newcommand{\RN}{{\mathbb R^N}}

\newcommand{\F}{{\mathcal F}}
\newcommand{\x}{{\rm x}}
\newcommand{\y}{{\rm y}}
\newcommand{\z}{{\rm z}}

\title[Fractional Leibniz rules for the Dunkl Laplacian ...]{Fractional Leibniz rules for the Dunkl Laplacian in Besov and Triebel--Lizorkin spaces}

\author[THE ANH BUI]{The Anh Bui}
\address{School of Mathematical and Physical Sciences, Macquarie University, NSW 2109,
	Australia}
\email{the.bui@mq.edu.au}

\author[XUETING HAN]{Xueting Han}
\address{Department of Mathematics, Hefei Institute of Technology, Hefei 238706, Anhui, China}
\email{hanxueting12@163.com}

\author[SUMAN MUKHERJEE]{Suman Mukherjee$^\dagger$}
\address{Department of Mathematics, Indian Institute of Technology Bombay, Powai, Mumbai--400076, India.}
\email{sumanmukherjee822@gmail.com}

\thanks{$\dagger$ \tt{Corresponding author}}

\begin{document}
	
	\begin{abstract}
Let $L$ be the Dunkl Laplacian on the Euclidean space $\mathbb{R}^N$ associated with a normalized root system $R$ and a multiplicity function $k(\nu)\geq 0$, $\nu\in R$. We establish a Leibniz-type rule for the fractional powers of $L$ on Besov and Triebel--Lizorkin spaces in the Dunkl setting. Our approach exploits the interplay between spectral multipliers and the Dunkl transform, together with the support properties of the distributions associated with Dunkl translations. These results extend the corresponding Leibniz-type estimates previously established on $L^p$ spaces to the broader setting of Besov and Triebel--Lizorkin spaces.

	\end{abstract}
	\date{}

	\maketitle
	
	%\tableofcontents
	
	\section{Introduction and main results}\label{sec: intro}

\subsection{Background}

At the core of many nonlinear partial differential equations, such as the Navier--Stokes, Euler, and Korteweg--de Vries equations, lies the fundamental problem of understanding the interaction between fractional differentiation and products of functions (see, for example, \cite{KP, CW, KPV, GK}). This phenomenon is governed by the fractional Leibniz rule, commonly referred to as the Kato--Ponce inequality. Introduced by Kato and Ponce \cite{KP}, this inequality extends the classical Leibniz rule to the setting of fractional derivatives. More precisely, for $s>0$ and Schwartz functions $f$ and $g$, one has
\[
||I^{s}(fg)||_{L^{p}} \lesssim ||I^{s}f||_{L^{p_{1}}}||g||_{L^{p_{2}}} + ||f||_{L^{p_{1}}}||I^{s}g||_{L^{p_{2}}},
\]
where $1/p = 1/p_1 + 1/p_2$, and $I^s = D^s$ or $J^s$. Here, $D^s$ and $J^s$ denote the Riesz and Bessel potentials, respectively, defined in terms of the Fourier transform by
\[
\widehat{(D^s f)}(\xi) = \|\xi\|^s \widehat f(\xi)
\]
and
\[
\widehat{(J^s f)}(\xi) = (1+\|\xi\|^2)^{s/2} \widehat f(\xi).
\]
These inequalities have become indispensable tools in harmonic analysis and the study of nonlinear PDEs, particularly in the analysis of pointwise products and bilinear pseudodifferential operators across a broad range of function spaces. In recent years, considerable attention has been devoted to extending such estimates beyond the Lebesgue spaces, particularly to nonclassical parameter regimes and subtle endpoint cases; see, for example, \cite{GO,BL}. Furthermore, related boundedness results have been extensively investigated in Sobolev, Besov, Triebel--Lizorkin, and other closely related function spaces \cite{BN,NT, HTW}.

The classical approach to establishing fractional Leibniz rules relies fundamentally on Fourier analytic techniques. As a consequence, these methods are inherently tied to the Euclidean setting, significantly limiting their applicability in more general geometric contexts. However, many problems arising in modern harmonic analysis and mathematical physics naturally occur in non-Euclidean frameworks, such as Lie groups or settings involving differential operators with nontrivial potentials. In these situations, classical Fourier methods are no longer directly applicable, making the extension of fractional Leibniz rules beyond the Euclidean setting a central problem in the field. To overcome these limitations, several efforts have been made to extend fractional Leibniz rules to broader non-Euclidean contexts. In particular:
\begin{itemize}
    \item \textbf{Lie groups:} Fractional Leibniz inequalities have been extensively studied in the context of nilpotent and stratified Lie groups, with corresponding results established for Lebesgue, Sobolev, Besov, and Triebel--Lizorkin spaces; we refer to \cite{BBR, Br, CRTN} for further details.
    
  \item \textbf{Specialized operators:} Analogous fractional Leibniz-type estimates have likewise been obtained for several important operators, including the Hermite and Grushin operators; see \cite{Br, NL} and the references therein.
    
  \item \textbf{Spaces of homogeneous type:} Fractional Leibniz-type inequalities have also been investigated on spaces of homogeneous type, particularly within Hardy space settings; related developments can be found in \cite{LZ}.
\end{itemize}
Despite these developments, the available results remain limited in scope. In most cases, the corresponding estimates are valid only under restrictive assumptions on the smoothness, integrability, and summability parameters. A notable recent development in this direction was presented in \cite{B2}, where the author established a broad and unified framework for fractional Leibniz rules far beyond the classical Euclidean setting to the general framework of spaces of homogeneous type. Departing entirely from Fourier-analytic techniques, the author introduced a novel and highly flexible approach based on bilinear spectral multiplier theory in conjunction with heat semigroup methods, thereby providing the first systematic and unified treatment of such inequalities in non-Euclidean environments. More precisely, the work establishes robust bilinear estimates for spectral multipliers associated with nonnegative self-adjoint operators acting on weighted Hardy, Besov, and Triebel--Lizorkin spaces over spaces of homogeneous type without imposing any restrictions on the underlying regularity or integrability parameters. Owing to the generality of the underlying framework, the results apply not only to the classical Euclidean setting but also to a broad range of settings, including nilpotent Lie groups, Hermite expansions, and Grushin-type operators. In addition, these developments furnish powerful tools for the analysis of nonlinear PDEs associated with the previously mentioned operators, particularly in connection with regularity theory and scattering phenomena, thereby considerably enlarging the scope and applicability of fractional Leibniz-type estimates.

\subsection{The Dunkl setup}

We now turn our attention to a different setting. In recent years, analytic Dunkl theory has emerged as a powerful extension of classical Euclidean harmonic analysis and has become an active area of research. The theory originated in the work of Dunkl \cite{D2} and has since undergone significant development (see, for example, \cite{DJ,JL,R,R2,R3,R4, TX1, TX2, DH1, DH2, B, MP}). We begin by presenting some preliminary notations and results in Dunkl theory that will be used throughout the paper. Consider the Euclidean space $\mathbb R^N$ equipped with its standard inner product and corresponding norm. For  {the vector} $0\neq \nu \in \RN$, the reflection $\sigma_\nu$ across the hyperplane orthogonal to $\nu$, denoted $\nu^\perp$, is given explicitly by
\begin{equation}
	\label{eq-reflection}
	\sigma_\nu \x = \x - 2 \frac{\langle \x, \nu \rangle}{\|\nu\|^2} \nu.
\end{equation}

A finite set $R \subset \RN \setminus {0}$ is called a root system if $R \cap \mathbb{R}\nu = \{\pm \nu\}$ for every $\nu \in R$, and $R$ is invariant under all reflections $\sigma_\nu$, that is, $\sigma_\nu(R)=R$ for each $\nu \in R$. In this work, we focus on normalized reduced root systems, meaning that $\|\nu\|^2 = 2$ for every $\nu \in R$. The finite group $G$ generated by the reflections  {$\{\sigma_\nu: \nu \in R\}$} is known as the Weyl group (or reflection group) corresponding to the root system $R$. A multiplicity function is a $G$-invariant map $k: R \to \mathbb{C}$, which will remain fixed and nonnegative throughout this paper. For a point $\x \in \RN$, we denote its $G$-orbit by 
\[
\mathcal O(\x) = \{\sigma(\x) : \sigma \in G\}, 
\]
and for a set $E \subset \RN$, its $G$-orbit is 
\[
\mathcal O(E) = \bigcup_{\x \in E} \mathcal O(\x).
\]
We also consider the $G$-invariant measure
\[
dw(\x) = \prod_{\nu \in R} |\langle \x,\nu\rangle|^{k(\nu)} \, d\x.
\]
The quantity 
\[
\mathfrak{N} := N + \sum_{\nu \in R} k(\nu)
\] 
is referred to as the homogeneous dimension of the system, as it satisfies
\begin{equation}
	\label{eq-homgoenous measure}
	w(B(t\x, tr)) = t^{\mathfrak{N}} w(B(\x,r)) \quad \text{for} \ \x \in \RN, \ t,r>0,
\end{equation}
where $B(\x,r) = \{\y \in \RN : \|\x - \y\| \le r\}$ denotes the closed Euclidean ball centered at $\x$ with radius $r>0$. It follows that
\begin{equation}
	\label{eq-volume formula}
	w(B(\x,r)) \simeq r^N \prod_{\nu \in R} (|\langle \x, \nu \rangle| + r)^{k(\nu)} \gtrsim r^{\mathfrak{N}},
\end{equation}
implying that $dw$ is a doubling measure. In other words, there exists a constant $C > 0$ such that
\begin{equation}
	\label{eq-doubling}
	w(B(\x, 2r)) \le C \, w(B(\x,r)) \quad \text{for} \ \x \in \RN, \ r>0.
\end{equation}
Furthermore, from \eqref{eq-volume formula}, one obtains
\begin{equation}
	\label{eq-ratios on volumes of balls}
	\Big(\frac{R}{r}\Big)^N \lesssim \frac{w(B(\x,R))}{w(B(\x,r))} \lesssim \Big(\frac{R}{r}\Big)^{\mathfrak{N}} \quad \text{for} \ \x \in \RN, \ 0<r<R.
\end{equation}
Next, we define the distance between two $G$-orbits $\mathcal O(\x)$ and $\mathcal O(\y)$ by
\begin{equation}\label{eq- d distance}
	d(\x,\y) = \min_{\sigma \in G} \|\x - \sigma(\y)\|.
\end{equation}
It is immediate that
\[
\mathcal O(B(\x,r)) = \{\y \in \RN : d(\x,\y) < r\},
\]
and moreover, for all $\x \in \RN$ and $r>0$,
\[
w(B(\x,r)) \le w(\mathcal O(B(\x,r))) \le |G| \, w(B(\x,r)).
\]
We will also write
\[
B^d(\x,r) = \{\y \in \RN : d(\x,\y) < r\}, \quad \text{so that} \quad B^d(\x,r) = \mathcal O(B(\x,r)).
\]
Let $E(\x,\y)$ denote the Dunkl kernel introduced in \cite{D3}. This kernel extends uniquely to a holomorphic function on $\mathbb C^N \times \mathbb C^N$. For any real number $0<p<\infty$, we denote by $L^p(dw)$ the space of measurable functions $f$ for which
\[
\|f\|_{L^p(dw)} := \Big(\int_{\RN} |f(\x)|^p \, dw(\x)\Big)^{1/p} < \infty. 
\] 
The Dunkl transform of a $L^1(dw)$ function $f$ is defined as
\begin{equation}
	\label{eq-Dunkl transform}
	\mathcal F f(\xi) = c_k^{-1} \int_{\mathbb R^N} E(-i\xi, \x) f(\x) \, dw(\x),
\end{equation}
where
\[
c_k = \int_{\mathbb R^N} e^{-  {\|\x\|^2/2}} dw(\x) > 0.
\]
It is well-known that the Dunkl transform extends to an isometry on $L^2(dw)$ and preserves the Schwartz space $\mathscr S(\RN)$. Its inverse is given by
\[
\F^{-1} g(\x) = c_k^{-1} \int_{\RN} E(i\xi, \x) g(\xi) \, dw(\xi).
\]
For further details, we refer the reader to \cite{DJ, R4}.

For $\x \in \RN$, the Dunkl translation operator $\bm{\tau}_\x$ is defined by
\begin{equation}\label{eq-translation}
	 \bm{\tau}_\x f(\y) = c_k^{-1} \int_{\RN} E(i\xi, \x) E(i\xi, \y) \, \F f(\xi) \, dw(\xi),
\end{equation}
and it is bounded on $L^2(dw)$. Moreover,  {for $\x,\y \in \RN$,} it is proved in \cite{Trim} that there exists a distribution $\gamma_{\x,\y}$ such that
\[
\bm{\tau}_\x f(\y) = \langle f, \gamma_{\x,\y} \rangle.
\] 
In addition, \cite[Theorem 5.1]{AAS} shows that the support of $\gamma_{\x,\y}$ is contained in the spherical shell
\begin{equation}
	\label{eq-support of gamma}
	\{ z \in \RN :  {|\|\x\| - \|\y\|| \le \|\z\| \le \|\x\| + \|\y\| }\} .
\end{equation}
We adopt the shorthand
\[
g(\x,\y) = \bm{\tau}_\x g(-\y) = \bm{\tau}_{-\y} g(\x).
\]
The Dunkl convolution of two suitable functions is defined by (see \cite{TX1})
\begin{equation}\label{eq-convolution defn}
	f * g(\x) = c_k \F^{-1}[\F f \, \F g](\x) = \int_{\RN} \F f(\xi) \, \F g(\xi) \, E(\x, i\xi) \, dw(\xi),
\end{equation}
or equivalently,
\begin{equation}\label{eq-defn f star g}
	f * g(\x) = \int_{\RN} f(\y) \, \bm{\tau}_\x g(-\y) \, dw(\y) = \int_{\RN} f(\y) \, g(\x, \y) \, dw(\y).
\end{equation}

We next define the Dunkl operators $T_\xi$, acting on a function $f$, as
\[
T_\xi f(\x) = \partial_\xi f(\x) + \sum_{\nu \in R} \frac{k(\nu)}{2} \langle \nu, \xi \rangle \frac{f(\x) - f(\sigma_\nu(\x))}{\langle \nu, \x \rangle},
\]
where $\partial_\xi f$ is the directional derivative of $f$ in the direction of the vector $\xi$. For the canonical basis $\{e_1, \dots, e_N\}$ of $\RN$, we set $T_j = T_{e_j}$. The Dunkl Laplacian is given by
\[
\begin{aligned}
	L &:= -\sum_{j=1}^N T_j^2 \\
	  &= - \Delta_{\RN} - \sum_{\nu \in R} k(\nu) \, \delta_\nu f(\x),
\end{aligned}
\]
where $\Delta_{\RN}$ denotes the usual Laplacian on $\RN$, and
\[
\delta_\nu f(\x) = \frac{\partial_\nu f(\x)}{\langle \nu, \x \rangle} - \frac{f(\x) - f(\sigma_\nu(\x))}{\langle \nu, \x \rangle^2}.
\]

It is known that $L$ defines a non-negative, self-adjoint operator on $L^2(dw)$ and generates a semigroup $e^{-tL}$ with kernel
\[
h_t(\x,\y): = c_k^{-1} (2t)^{-\mathfrak{N}/2} \exp\Big(- \frac{\|\x\|^2 + \|\y\|^2}{4t}\Big) E\Big(\frac{\x}{\sqrt{2t}}, \frac{\y}{\sqrt{2t}}\Big), \quad \x, \y \in \RN, \ t>0.
\]
It is well-known that 
\[
\int_{\RN}h_t(\x,\y)dw(\y)=\int_{\RN}h_t(\y,\x)dw(\y)=1
\]
for all $\x\in \RN$ and $t>0$. In addition,  there are constants $C, c>0$ such that
\begin{equation}\label{eq-Gaussian upper}
h_t(\x,\y)\le C\f{1}{w(B(\x,\sqrt t))}\exp\Big(-\f{d(\x,\y)^2}{ct}\Big)
\end{equation}
for all $\x,\y\in \RN$ and $t>0$. For further details, we refer to \cite{AH,R,DH2}.

\subsection{Main results}

Owing to the rapid developments in analytic Dunkl theory, it is natural to investigate nonlinear equations associated with the Dunkl Laplacian, including the Euler and Navier--Stokes systems, as well as dispersive equations such as the Korteweg--de Vries equation in the Dunkl setting. The Dunkl Laplacian also comes into play in the study of smoothing properties of Schr\"odinger semigroups, regularity theory, and scattering phenomena within the Dunkl framework. These directions have recently attracted considerable attention (see, for example, \cite{AAS, Hej}) and are expected to remain an active area of research in the coming years. In this context, the study of fractional Leibniz inequalities associated with the Dunkl Laplacian arises naturally. As in several other settings, fractional Leibniz inequalities for the Dunkl Laplacian were previously investigated in \cite{BM}, although only in the framework of Lebesgue spaces. Motivated by recent advances concerning fractional Leibniz inequalities for nonnegative self-adjoint operators on Besov and Triebel--Lizorkin spaces over spaces of homogeneous type \cite{B2}, it is natural to seek analogous results in the Dunkl setting within the broader class of Besov and Triebel--Lizorkin spaces.

At first glance, as in the case of self-adjoint operators on Lie groups, the Hermite operators, or the Grushin operators, one might expect such estimates to follow directly from the general theory developed in \cite{B2}. However, this is not the case. Indeed, the assumptions of H\"older continuity and Gaussian upper bounds for higher-order derivatives of the associated heat kernel, required in \cite{B2}, do not hold in the Dunkl framework. Consequently, the existing general theory cannot be applied directly to obtain the desired results. The primary objective of the present work is therefore to establish corresponding fractional Leibniz rules for the Dunkl Laplacian within the broader functional setting of Besov and Triebel--Lizorkin spaces, in the full range of parameters, without restrictions on the smoothness index $s$ or on the integrability and summability parameters $p$ and $q$.

In what follows, $\dot B^{s,L}_{p,q}(dw)$, $\dot F^{s,L}_{p,q}(dw)$ and $H^p_L(dw)$ denote the Besov, Triebel--Lizorkin, and Hardy spaces associated with the Dunkl Laplacian $L$, respectively. See Section \ref{subsec: Besov and other spaces} for precise definitions of these function spaces. Our main result concerns the fractional Leibniz rule for the Dunkl Laplacian $L$ on Besov and Triebel--Lizorkin spaces, which is stated as follows:

\begin{thm}\label{main thm}
	Let $0<q\le \vc$, $0 < p,p_1,p_2,p_3,p_4 \leq \infty$ with 
	\begin{equation}\label{eq- pi condition}
		\frac{1}{p} = \frac{1}{p_1} + \frac{1}{p_2} = \frac{1}{p_3} + \frac{1}{p_4}.
	\end{equation}

\noindent If $0<p_1,p_4<\vc$, $0<p_2,p_3\le \vc$, and $s > \mathfrak{N}\Big(\f{1}{\min\{p,q,1\}}-1\Big)$, then
	\begin{equation}\label{eq-main thm TL spaces}
		\|L^{s/2}(fg)\|_{\dot{F}^{0,L}_{p,q}(dw)} \simeq \|fg\|_{\dot{F}^{s,L}_{p,q}(dw)} 
		\lesssim \|f\|_{\dot{F}^{s,L}_{p_1,q}(dw)} \|g\|_{H^{p_2}_L(dw)} 
		+ \|f\|_{H^{p_3}_L(dw)} \|g\|_{\dot{F}^{s,L}_{p_4,q}(dw)}. 
	\end{equation}
	
	\medskip
	\noindent If $0 < p,p_1,p_2,p_3,p_4 \leq \infty$ and $s > \mathfrak{N}\Big(\f{1}{\min\{p,1\}}-1\Big)$, then
	\begin{equation}\label{eq-main thm B spaces}
		\|L^{s/2}(fg)\|_{\dot{B}^{0,L}_{p,q}(dw)}\simeq \|fg\|_{\dot{B}^{s,L}_{p,q}(dw)} 
		\lesssim \|f\|_{\dot{B}^{s,L}_{p_1,q}(dw)} \|g\|_{H^{p_2}_L(dw)} 
		+ \|f\|_{H^{p_3}_L(dw)} \|g\|_{\dot{B}^{s,L}_{p_4,q}(dw)}. 
	\end{equation}
	In both inequalities, $H^{p_i}_L(dw)$ is replaced by $L^\infty$ if $p_i=\infty$  for $i=2,3$.
\end{thm}
\begin{rem}
In view of the identities $H^p_L(dw)=\dot{F}^{0,L}_{p,2}(dw)$ for all $0<p<\infty$ and $H^p_L(dw)=L^p(dw)$ for $1<p<\infty$ (see Section \ref{subsec: Besov and other spaces}), the above theorem recovers the results of \cite{BM} as a particular case.
 \end{rem}
As mentioned earlier, Theorem \ref{main thm} cannot be derived from the general results in \cite[Corollary 1.2]{B2}, since the heat kernel associated with the Dunkl Laplacian fails to satisfy the H\"older continuity assumption assume in \cite{B2}. More importantly, assumption (A3) in \cite{B2}, which plays a crucial role in obtaining the full range of the smoothness index $s$ in the fractional Leibniz rules, does not hold in our setting. In order to overcome this difficulty, we employ the connection of the spectral multiplier and the Dunkl transform, and  the support of the distribution associated with the Dunkl translation as in \eqref{eq-support of gamma}.

We conclude this section with an overview of the rest of the paper. Section \ref{subsec:Maximal functions and related inequalities} collects results on maximal functions and related vector valued inequalities. In Section \ref{subsec: Dyadic cubes}, we present dyadic cubes and auxiliary estimates in Dunkl setting, followed by functional calculus and kernel estimates for the operator $L$ in Section \ref{subsec: functional calculus}. Section \ref{subsec: Besov and other spaces} introduces Besov, Triebel–Lizorkin, and Hardy spaces associated with $L$, together with essential properties used in later arguments. Finally, Section \ref{sec: proof of main thm} contains the proof of the main theorem.

\section{Auxiliary definitions and results}

In this section, we collect several auxiliary definitions and results which will be used throughout the paper. 

\subsection{Maximal functions and related inequalities}\label{subsec:Maximal functions and related inequalities}

Let $0<r<\infty$. The Hardy--Littlewood maximal operator $\mathcal{M}^{\mathcal O}_{r}$ is defined by
\begin{equation}\label{eq-maximal function}
	\mathcal{M}^{\mathcal O}_{r} f(\x)
	=\sup_{ { \mathcal O(B)\ni \x}}\Big(\frac{1}{w(\mathcal O(B))}\int_{\mathcal O(B)}|f(\y)|^r\,dw(\y)\Big)^{1/r},
\end{equation}
where the supremum is taken over all balls $B$  such that $\x\in \mathcal O(B)$. We write $\mathcal{M}^{\mathcal O}$ instead of $\mathcal{M}^{\mathcal O}_1$.  
It is well known that
\begin{equation}\label{boundedness maximal function}
	\|\mathcal{M}^{\mathcal O}_{r} f\|_{L^p(dw)}\lesssim \|f\|_{L^p(dw)}
\end{equation}
for all $p>r$.

We also recall the Fefferman--Stein vector-valued maximal inequality from \cite{GLY}.  
If $0<p<\infty$, $0<q\le \infty$, and $0<r<\min\{p,q\}$, then for any sequence $\{f_{ {i}}\}$ of measurable functions,  
\begin{equation}\label{FSIn}
	\Big\|\Big(\sum_{{ {i}}}|\mathcal{M}^{\mathcal O}_r f_{ {i}}|^q\Big)^{1/q}\Big\|_{L^p(dw)}
	\lesssim
	\Big\|\Big(\sum_{{ {i}}}|f_{ {i}}|^q\Big)^{1/q}\Big\|_{L^p(dw)}.
\end{equation}
Combining Young’s inequality with \eqref{FSIn}, we obtain that if $\{a_{ {i}}\}\in \ell^{q}\cap \ell^{1}$, then
\begin{equation}\label{YFSIn}
	\Big\|\sum_{j}\Big(\sum_{{i}}|a_{j-{{i}}}\mathcal{M}^{\mathcal O}_r f_{{i}}|^q\Big)^{1/q}\Big\|_{L^p(dw)}
	\lesssim
	\Big\|\Big(\sum_{{{i}}}|f_{{i}}|^q\Big)^{1/q}\Big\|_{L^p(dw)}.
\end{equation}

We record the following elementary estimate for later use (see, e.g., \cite{BBD}).

\begin{lem}\label{lem-elementary}
	Let $\epsilon>0$. For any $f\in L^1_{\mathrm{loc}}({\RN})$, $\x\in \mathbb{R}^N$, and $s>0$, we have
	\[
	\int_{\mathbb{R}^N}\frac{1}{w(B(\x,s))}\Big(1+\frac{d(\x,\y)}{s}\Big)^{-(\mathfrak{N}+\epsilon)}|f(\y)|\,d\y
	\lesssim \mathcal{M}^{\mathcal O}f(\x).
	\]
\end{lem}

\subsection{Dyadic cubes and auxiliary estimates}\label{subsec: Dyadic cubes}

A \emph{dyadic cube} in $\mathbb{R}^N$ is a cube of the form
\[
Q = 2^{-k}([m_1, m_1+1)\times\cdots\times [m_N, m_N+1)),
\]
where $k\in\mathbb{Z}$ and $m_1,\dots,m_N\in\mathbb{Z}$.  
We denote by $\x_Q=2^{-k}(m_1,\ldots,m_N)$ the lower-left corner of $Q$ and by $\ell(Q)=2^{-k}$ its side length.  
Let $\mathscr{D}=\{Q\}$ denote the family of all dyadic cubes, and for each $k\in\mathbb{Z}$, set
\[
\mathscr{D}_k=\{Q\in \mathscr D:\,\ell(Q)=2^{-k}\}.
\]
We next state the following key lemma, which provides a discrete molecular summation estimate by means of the maximal operator.
\begin{lem}[\cite{B}]\label{lem1- thm2 atom Besov}
	Let $n_0\in\mathbb{Z}$ and $\mathscr D=\{\mathscr D_k\}_{k\in\mathbb{Z}}$ be the dyadic cubes in $\mathbb{R}^N$.  
	Fix $M>\mathfrak{N}$, $\kappa\in [0,1]$, and $\eta,k\in\mathbb{Z}$ with $k\ge \eta$.  
	Assume that $\{f_Q\}_{Q\in\mathscr{D}_k}$ satisfies
	\begin{equation}\label{eq-bound of fQ}
		|f_{Q}(\x)|\lesssim
		\Big(\frac{w(Q)}{w(B(\x_Q,2^{-\eta}))}\Big)^\kappa
		\Big(1+\frac{d(\x,\x_Q)}{2^{-\eta}}\Big)^{-M}.
	\end{equation}
	Then, for $\mathfrak{N}/M<r\le 1$ and any sequence $\{s_Q\}_{Q\in \mathscr{D}_k}$,
	\[
	\sum_{Q\in \mathscr{D}_{k+n_0}}|s_Q||f_Q(\x)|
	\lesssim
	2^{\mathfrak{N}(k-\eta)(1/r-\kappa)}
	\mathcal{M}^{\mathcal O}_r\Big(\sum_{Q\in \mathscr{D}_{k+n_0}}|s_Q|\chi_Q\Big)(\x).
	\]
\end{lem}

\subsection{Functional calculus and kernel estimates}\label{subsec: functional calculus}

Since $L$ is a nonnegative self-adjoint operator on $L^2(dw)$,  
for each bounded Borel function $F:[0,\infty)\to\mathbb{C}$, we can define
\[
F(L)=\int_0^\infty F(\lambda)\,dE(\lambda),
\]
where $E(\lambda)$ is the spectral resolution of $L$. Clearly, $F(L)$ is bounded on $L^2(dw)$.

The next lemma establishes kernel estimates for the functional calculus of $L$. While the corresponding estimates without the additional decay factor
$
\left(1+t^{-1}|x-y|\right)^{-2}
$
are standard for general self-adjoint operators on spaces of homogeneous type, the presence of this additional decay relies on special structural properties of the operator $L$.

\begin{lem}[\cite{B}]\label{lem-heat kernel estimate}
	Let $\varphi\in\mathscr{S}(\mathbb{R})$ be an even function. Then:
	\begin{enumerate}[\rm(a)]
		\item For every $M>0$ and multi-indices $\alpha,\beta$, the kernel $\varphi(t\sqrt{L})(\x,\y)$ {of $\varphi(t\sqrt{L})$}  satisfies
		\[
		|\partial_{\x}^{\alpha}\partial_{\y}^{\beta}\varphi(t\sqrt{L})(\x,\y)|
		\lesssim
		t^{-(|\alpha|+|\beta|)}
		\Big(1+\frac{\|\x-\y\|}{t}\Big)^{-2}
		\frac{1}{w(B(\x,t+d(\x,\y)))}\Big(\frac{t}{t+d(\x,\y)}\Big)^{M}.
		\]
		\item For every $M>0$, if $\|\y-\y'\|\le t$, then
		\[
        \begin{aligned}
		 & |\varphi(t\sqrt{L})(\x,\y)-\varphi(t\sqrt{L})(\x,\y')|\\
        & \lesssim 
		\frac{\|\y-\y'\|}{t}
		\Big(1+\frac{\|\x-\y\|}{t}\Big)^{-2}
		\frac{1}{w(B(\x,t+d(\x,\y)))}\Big(\frac{t}{t+d(\x,\y)}\Big)^{M}.
         \end{aligned}
		\]
	\end{enumerate}
\end{lem}

\subsection{Besov, Triebel–Lizorkin, and Hardy spaces associated with $L$}\label{subsec: Besov and other spaces} Since $L$ is a nonnegative self-adjoint operator on $L^2(\mathbb{R}^N)$ and satisfies the Gaussian upper bound \eqref{eq-Gaussian upper}, following \cite{BBD}, we can define the Besov and Triebel–Lizorkin spaces associated with $L$. Interestingly, these function spaces coincide with the classical spaces, including the Besov, Triebel–Lizorkin, Lebesgue, and Hardy spaces defined on the space of homegeneous type $(\RN, \|\cdot\|, dw)$, in the sense of \cite{HMY, HS, CW}.

The class of test functions associated with $L$ is defined by
\[
\mathcal{S}
=\Big\{\phi\in\bigcap_{m\ge 1}D(L^m):~
\mathcal{P}_{m,\ell}(\phi)
 {:=\sup_{\x\in \RN}(1+\|\x\|)^m|L^\ell \phi(\x)|}<\infty
\quad\forall m>0,\;\ell\in\mathbb{N}\Big\}.
\]
It was shown in \cite{PK} that $\mathcal{S}$ is a complete locally convex space generated by the seminorms $\{\mathcal{P}_{m,\ell}\}$, and that for the operator $- \Delta_{\RN}$ on $\mathbb{R}^N$, $\mathcal{S}$ coincides with the Schwartz class $\mathscr{S}(\mathbb{R}^N)$. The dual space $\mathcal{S}'$ of $\mathcal{S}$, consisting of continuous linear functionals on $\mathcal{S}$, is equipped with the pairing
\[
\langle f,\phi\rangle=f(\overline{\phi}),\qquad f\in\mathcal{S}',~\phi\in\mathcal{S}.
\]

For the homogeneous theory of Besov and Triebel–Lizorkin spaces associated with $L$, we follow \cite{G.etal} and define
\[
\mathcal{S}_\infty=\{\phi\in \mathcal{S}:\,\forall k\in\mathbb{N},~\exists\, g_k\in \mathcal{S}\text{ such that }\phi=L^k g_k\},
\]
whose topology is generated by seminorms $\mathcal{P}^*_{m,\ell,k}(\phi)=\mathcal{P}_{m,\ell}(g_k)$.

A \emph{partition of unity} is a function $\psi\in \mathscr{S}(\mathbb{R})$ satisfying 
$\supp \psi\subset [1/2,2]$, $$\int \psi(\xi)\frac{d\xi}{\xi}\neq 0,$$ and
\[
\sum_{j\in \mathbb{Z}}\psi_j(\lambda)=1,\qquad \lambda>0,
\]
where $\psi_j(\lambda)=\psi(2^{-j}\lambda)$.
With the above preparations in place, we now introduce the Besov and Triebel–Lizorkin spaces associated with the Dunkl Laplacian in the spirit of \cite{BBD}.

\begin{defn}
	Let $\psi$ be a partition of unity. For $0<p,q\le\infty$ and $s\in\mathbb{R}$, define
	\[
	\|f\|_{\dot{B}^{s,\psi,L}_{p,q}(dw)}
	=\Big(\sum_{j\in\mathbb{Z}}\big(2^{js}\|\psi_j(\sqrt{L})f\|_{L^p(dw)}\big)^q\Big)^{1/q}.
	\]
	The homogeneous Besov space $\dot{B}^{s,\psi,L}_{p,q}(dw)$ is the set of all $f\in \mathcal{S}'_\infty$ with $\|f\|_{\dot{B}^{s,\psi,L}_{p,q}(dw)}<\infty$.  
	Similarly,  for Triebel–Lizorkin space, we define
	\[
	\|f\|_{\dot{F}^{s,\psi,L}_{p,q}(dw)}
	=\Big\|\Big(\sum_{j\in\mathbb{Z}}(2^{js}|\psi_j(\sqrt{L})f|)^q\Big)^{1/q}\Big\|_{L^p(dw)},
	\]
	and the homogeneous Triebel--Lizorkin space $\dot{F}^{s,\psi,L}_{p,q}(dw)$ is defined by
\[
\dot{F}^{s,\psi,L}_{p,q}(dw)
=
\Big\{f\in \mathcal{S}'_\infty:
\|f\|_{\dot{F}^{s,\psi,L}_{p,q}(dw)}<\infty\Big\}.
\]
\end{defn}

Since these spaces do not depend on the choice of $\psi$ (see \cite{BBD}), we henceforth denote them by $\dot{B}^{s,L}_{p,q}(dw)$ and $\dot{F}^{s,L}_{p,q}(dw)$ instead of $\dot{B}^{s,\psi,L}_{p,q}(dw)$ and $\dot{F}^{s,\psi,L}_{p,q}(dw)$, respectively.

For later reference, we recall the following proposition from \cite{BBD}, which shows that fractional powers of $L$ continuously map Besov and Triebel–Lizorkin spaces into spaces of higher smoothness within the same scale.

\begin{prop}\label{prop- Ls on TL B spaces}
	Let $s,\alpha\in\mathbb{R}$. Then
	\[
	\|L^{s/2}f\|_{\dot{B}^{\alpha,L}_{p,q}(dw)}\simeq \|f\|_{\dot{B}^{s+\alpha,L}_{p,q}(dw)}, 
	\quad 0<p,q\le\infty,
	\]
	and
	\[
	\|L^{s/2}f\|_{\dot{F}^{\alpha,L}_{p,q}(dw)}\simeq \|f\|_{\dot{F}^{s+\alpha,L}_{p,q}(dw)}, 
	\quad 0<p<\infty,~0<q\le\infty.
	\]
\end{prop}

For $\lambda>0$, $j\in\mathbb{Z}$, and $\varphi\in\mathscr{S}(\mathbb{R})$, the Peetre-type maximal functions are defined by
\begin{align}
	\label{eq1-PetreeFunction}
	\varphi_{j,\lambda}^*(\sqrt{L})f(\x)
	&=\sup_{\y\in \mathbb R^N}\frac{|\varphi_j(\sqrt{L})f(\y)|}{(1+2^j d(\x,\y))^{\lambda}},\\
	\label{eq2-PetreeFunction}
	\varphi_{\lambda}^*(s\sqrt{L})f(\x)
	&=\sup_{\y\in \mathbb R^N}\frac{|\varphi(s\sqrt{L})f(\y)|}{(1+d(\x,\y)/s)^{\lambda}}.
\end{align}
The next proposition gives equivalent characterizations of the Besov and Triebel--Lizorkin spaces introduced above in terms of the Peetre-type maximal functions.

\begin{prop}[\cite{BBD}]\label{prop2-thm1}
	Let $\psi$ be a partition of unity. Then:
	\begin{enumerate}[\rm(a)]
		\item For $0<p,q\le\infty$, $s\in\mathbb{R}$ and $\lambda>\N/p$,
		\[
		\Big(\sum_{j\in\mathbb{Z}}(2^{js}\|\psi^*_{j,\lambda}(\sqrt{L})f\|_{L^p(dw)})^q\Big)^{1/q}
		\simeq
		\|f\|_{\dot{B}^{s,L}_{p,q}(dw)}.
		\]
		\item For $0<p<\infty$, $0<q\le\infty$, $s\in\mathbb{R}$ and $\lambda>\max\{\N/q, \N/p\}$,
		\[
		\Big\|\Big(\sum_{j\in\mathbb{Z}}(2^{js}|\psi^*_{j,\lambda}(\sqrt{L})f|)^q\Big)^{1/q}\Big\|_{L^p(dw)}
		\simeq
		\|f\|_{\dot{F}^{s,L}_{p,q}(dw)}.
		\]
	\end{enumerate}
\end{prop}

For $0<p\le \infty$, we define the Hardy space  $H^p_L(dw)$ as the set of all functions $f\in L^2(dw)$ such that 
\[
\|f\|_{H^p_L(dw)}:=\Big\|\sup_{t>0}|e^{-tL}f|\Big\|_{L^p(dw)}<\vc.
\]
 It was proved in \cite{BBD} that $H^p_L(dw)=\dot{F}^{0,L}_{p,2}(dw)$ for all $0<p<\infty$, and that $H^p_L(dw)=L^p(dw)$ for $1<p<\infty$. For the Hardy spaces as well, we provide in the next proposition an equivalent characterization in terms of Peetre-type maximal functions.

\begin{lem}\label{lem-Hardy}
	Let $\varphi\in \mathscr{S}(\mathbb{R})$ be even and satisfy $\varphi(0)\ne 0$.  
	Then for $0<p\le\infty$ and $\lambda>2\mathfrak{N}/\min\{p,1\}$,
	\[
	\Big\|\sup_{s>0} |\varphi^*_\lambda(s\sqrt L)f|\Big\|_{L^p(dw)}
	\lesssim \|f\|_{H^p_{L}(dw)}.
	\]
	When $p=\infty$, the Hardy space $H^p_{L}(dw)$ is replaced by $L^\infty(dw)$.
\end{lem}

\begin{proof}
	The proof for $1<p\le\infty$ is standard and omitted.  
	The case $0<p\le 1$ can be found in \cite[Theorem~3.7]{DKKP}.
\end{proof}

\section{Proof of the fractional Leibniz rules}\label{sec: proof of main thm}

Before proving Theorem \ref{main thm}, we establish two technical lemmas that will be needed in the proof.

\begin{lem}\label{lem1}
Let $A>0$ and let $\psi, \phi$ and $\varphi$ be smooth functions such that 
$$
\supp\psi \subset [A/2,2A], \quad\supp\phi\subset [0,2A]
$$ 
and
$$
\supp\varphi \subset [0,5A], \quad\varphi=1 \ \ \text{on} \ \ [0,4A].
$$ 
Then we have
\[
\varphi(\sqrt L)\big[\psi(\sqrt L)f\phi(\sqrt L)g\big]=\psi(\sqrt L)f\phi(\sqrt L)g.
\]
\end{lem}
\begin{proof}
 {By \cite{B}, we know that
$$ \varphi(\sqrt L)f= \mathcal F^{-1}(\varphi(\|\cdot\|)\mathcal F f). $$
}
	Then, we have
	\begin{equation}\label{eq1-support}
	\mathcal F\Big\{\varphi(\sqrt L)\big[\psi(\sqrt L)f\phi(\sqrt L)g\big]\Big\}(\xi)=\varphi(\|\xi\|)\mathcal F\big[\psi(\sqrt L)f\phi(\sqrt L)g\big](\xi).
	\end{equation}
	On the other hand,
	\[
	\psi(\sqrt L)f\phi(\sqrt L)g= \mathcal F^{-1}\big[\psi(\| {\cdot}\|)\mathcal F f\big]\mathcal F^{-1}\big[\phi(\| {\cdot}\|)\mathcal F g\big].
	\]
	It follows that
	\[
	\begin{aligned}
		\mathcal F\big[\psi(\sqrt L)f\phi(\sqrt L)g\big]&=\big[\psi(\|\cdot\|)\mathcal F f\big]\ast\big[\phi(\|\cdot\|)\mathcal F g\big].
	\end{aligned}
	\]
	From \eqref{eq-support of gamma}, \eqref{eq-defn f star g} and the support conditions of $\psi$ and $\phi$, we have 
	\[
	\supp \big[\psi(\|\cdot\|)\mathcal F f\big]\ast\big[\phi(\|\cdot\|)\mathcal F g\big] \subset [0,4A];
	\]
	and hence,
	\[
	\varphi(\|\xi\|)\mathcal F\big[\psi(\sqrt L)f\phi(\sqrt L)g\big](\xi)=\mathcal F\big[\psi(\sqrt L)f\phi(\sqrt L)g\big](\xi).
	\]
	This, together with \eqref{eq1-support}, implies
	\[
	\mathcal F\Big\{\varphi(\sqrt L)\big[\psi(\sqrt L)f\phi(\sqrt L)g\big]\Big\}(\xi)=\mathcal F\big[\psi(\sqrt L)f\phi(\sqrt L)g\big](\xi).
	\]
	Consequently,
	\[
 \varphi(\sqrt L)\big[\psi(\sqrt L)f\phi(\sqrt L)g\big] = \psi(\sqrt L)f\phi(\sqrt L)g.
	\]
	This completes our proof.
\end{proof}
\begin{lem}\label{lem-kernel-representation}  
For $\lambda, n>0$, we set 
\[D_{\lambda, n}(\x, \y)= \frac{1}{w(B(\x, \lambda))}\left(1+ \frac{d(\x,\y)}{\lambda}\right)^{-n}.\]
	 Then there exist an universal constant $k_0\in \mathbb Z^+$ and  a sequence of measurable functions $\{\widetilde\Gamma_j\}_{j\in \mathbb Z}$ defined on $\RN\times \RN$ satisfying the following properties:
	\begin{enumerate}[\rm (i)]
		\item For every  { $m\in \mathbb N$, there exists $C_{m}>0$} such that for all $\x,\y\in \RN$ and $j\in \mathbb Z$,
		\[
		|L^m\widetilde\Gamma_j(\x,\y)|\le C_{m}\,2^{2jm} D_{2^{-j},M}(\x,\y).
		\]
		
		\item For   $j\in \mathbb Z$ and $M>0$,
		\[
		|\widetilde\Gamma_j(\x,\y)-\widetilde\Gamma_j(\overline \x,\y)|
		\le  C  (2^jd(\x,\overline \x) )  D_{2^{-j},M}(\x,\y),
		\quad \text{whenever } d(\x,\overline \x)<2^{-j}.
		\] 
		\item For $\psi \in \mathscr{S}([0,\infty))$ with $\supp \psi \subset [0,2]$ and any $j\in \mathbb Z$, 
		\[
		\psi_j(\sqrt L)f(\x) 
		= \sum_{Q\in \mathscr D_{j+k_0}}\omega_Q \,\psi_j(\sqrt L)f(\x_Q)\,\widetilde\Gamma_j(\x,\x_Q) 
		\]
		for $f\in \mathcal S_{\vc}, $ where $\omega_Q\simeq w(Q)$ for each $Q\in \mathscr D_j$ and each $j\in \mathbb N$.
	\end{enumerate}
\end{lem}
\begin{proof}
	The proof of this lemma is similar to that of \cite[Lemma 4.2]{PK}. 
	It suffices to prove the following estimate: let $\Gamma \in \mathscr{S}(\mathbb{R})$ be even and supported in $[-3,3]$, and let $j_0 \in \mathbb{Z}^+$ be a fixed positive integer. 
	For $j \in \mathbb{Z}$ and $1 \leq p < \infty$, we have
	\begin{equation} \label{eq:4.3}
		\sum_{Q \in \mathscr{D}_{j+j_0}} \int_{Q} 
		|\Gamma_j(\sqrt{L})f(\x) - \Gamma_j(\sqrt{L})f(\x_Q)|^p \, dw(\x) 
		\leq C 2^{-j_0 p}\|\Gamma_j(\sqrt{L})f\|_{L^p(dw)}^p.
	\end{equation}
Note that the estimate \eqref{eq:4.3} does not follow from \cite{PK}, since in our setting we do not have the H\"older continuity of the heat kernel of $L$. Once the estimate \eqref{eq:4.3} has been proved, we can follow the proof of \cite[Lemma 4.2]{PK} to obtain the lemma. 
	
	To prove \eqref{eq:4.3}, let $\varphi \in \mathscr{S}(\mathbb{R})$ be even such that 
	$\supp \varphi \subset [-4,4]$ and $\varphi = 1$ on $[-3,3]$. 
	Then we have
	\[
	\Gamma_j(\sqrt{L})f(\x) 
	= \int_{\mathbb{R}^N} \varphi_j(\sqrt{L})(x, y)[\Gamma_j(\sqrt{L})f](\y)\, dw(\y).
	\]
	Using Lemmas \ref{lem-heat kernel estimate} and \ref{lem-elementary}, we obtain
	\[
	\begin{aligned}
		\sum_{Q \in \mathscr{D}_{j+j_0}} \int_{Q} 
		&|\Gamma_j(\sqrt{L})f(\x) - \Gamma_j(\sqrt{L})f(\x_Q)|^p \, dw(\x) \\
		&= \sum_{Q \in \mathscr{D}_{j+j_0}} \int_{Q} 
		\left| \int_{\mathbb{R}^N} 
		\left[\varphi_j(\sqrt{L})(x, y) - \varphi_j(\sqrt{L})(x_Q, y)\right] 
		[\Gamma_j(\sqrt{L})f](\y) \, dw(\y) \right|^p dw(\x) \\
		&\le C \sum_{Q \in \mathscr{D}_{j+j_0}} \int_{Q} 
		\left( \int_{\mathbb{R}^N}  
		\frac{d(\x,\x_Q)}{2^{-j}} 
		D_{2^{-j},\mathfrak{N}+1}(\x,\y)|\Gamma_j(\sqrt{L})f(\y)| \, dw(\y) \right)^p dw(\x) \\
		&\le C 2^{-j_0 p}\sum_{Q \in \mathscr{D}_{j+j_0}} 
		\int_{Q}(\mathcal{M}^{\mathcal{O}}|\Gamma_j(\sqrt{L})f|(\x))^p\, dw(\x)\\
        &
		= C 2^{-j_0 p}\|\mathcal{M}^{\mathcal{O}}|\Gamma_j(\sqrt{L})f|\|_{L^p(dw)}^p \\
		&\le C 2^{-j_0 p}\|\Gamma_j(\sqrt{L})f\|_{L^p(dw)}^p.
	\end{aligned}
	\]
	This completes the proof.	
\end{proof}

We are now ready to present the proof of Theorem \ref{main thm}.
	
\begin{proof}[Proof of Theorem \ref{main thm}:] We only give the proof of \eqref{eq-main thm TL spaces} since the proof of \eqref{eq-main thm B spaces} can be done similarly.
	
	\medskip
	
We now fix  $0 < p,p_1, p_4 < \infty$ and $0 <  p_2,p_3  \le \infty$ such that $1/p = 1/p_1 + 1/p_2=1/p_3 + 1/p_4$ and $0 < q \le \infty$.

 Fix \[
	s>\tau_{p,q} := \mathfrak{N} \left( \frac{1}{\min(p,q,1)} - 1 \right).\]
%Recall that $\tau_p(dw) := n \left( \frac{1}{\min(p/\tau_w,1)} - 1 \right).$

We also fix $0<r<\min\{1,p,q\}$ such that $s>\mathfrak N(1/r-1)$.

Let $\psi\in C^\infty(\mathbb{R})$ be  a partition of unity  such that $0\le\psi\le1$ and
\[
\supp\psi(\xi)\subset [1/2,2],\qquad \sum_{j\in\mathbb Z}\psi_j(\xi) = 1 \ \text{for} \ \xi>0,
\]
where $\psi_j(\cdot)=\psi(2^{-j}\cdot)$.

Set
\[
\phi(\xi) =\begin{cases}
	\sum\limits_{j\le 0}\psi_j(\xi), \ \ & \xi>0\\
	1, \ \ & \xi=0.
\end{cases}
\]
It is easy to see that $\phi\in \mathscr{S}([0,\infty))$ and $\supp \phi\subset [0,2]$; moreover, for any $k\in \mathbb Z$, 
\[
\phi_k(\xi):=\phi(2^{-k}\xi) = \sum_{j\le k}\psi_j(\xi). 
\]

We write
\begin{equation}\label{eq- - m decomposition}
\begin{aligned}
	fg&=\sum_{k\in \mathbb Z}\sum_{j\in \mathbb{Z}}  \psi_k(\sqrt L)f\psi_j(\sqrt L)g\\
	&=\sum_{k\in \mathbb Z}\sum_{j\in \mathbb Z\atop  j\le  k}  \psi_k(\sqrt L)f\psi_j(\sqrt L)g +\sum_{j\in \mathbb Z}\sum_{k\in \mathbb Z \atop k\le j-1}  \psi_k(\sqrt L)f\psi_j(\sqrt L)g\\
	&= \sum_{k\in \mathbb Z}   \psi_k(\sqrt L)f\phi_k(\sqrt L)g +\sum_{j\in \mathbb Z}   \phi_{j-1}(\sqrt L)f\psi_j(\sqrt L)g.
\end{aligned}
\end{equation}
These two terms are similar and hence it suffices to estimate the first term. To do this, we first observe that $\supp \psi_k\subset [2^{k-1}, 2^{k+1}]$ and $\supp \phi_k\subset [0, 2^{k+1}]$. This, together with the fact $\phi_{k+3} =1 $ on $[0,2^{k+2}]$ and Lemma \ref{lem1}, gives
\[
\psi_k(\sqrt L)f\phi_k(\sqrt L)g=\phi_{k+3}(\sqrt L)[\psi_k(\sqrt L)f\phi_k(\sqrt L)g].
\]

Hence, applying Lemma \ref{lem-kernel-representation},
\[
\begin{aligned}
	\sum_{k\in \mathbb Z}   \psi_k(\sqrt L)f\phi_k(\sqrt L)g(\x) &=\sum_{k\in \mathbb Z}\phi_{k+3}(\sqrt L)[\psi_k(\sqrt L)f\phi_k(\sqrt L)g]\\
	&= \sum_{k\in \mathbb Z} \ \sum_{Q\in \mathscr D_{k+k_0+3}} \omega_Q\phi_{k+3}(\sqrt L)[\psi_k(\sqrt L)f\phi_k(\sqrt L)g](\x_Q) \widetilde{\Gamma}_k(\x,\x_Q)\\
	&= \sum_{k\in \mathbb Z} \ \sum_{Q\in \mathscr D_{k+k_0+3}}  \omega_Q \psi_k(\sqrt L)f(\x_Q)\phi_k(\sqrt L)g(\x_Q) \widetilde{\Gamma}_k(\x,\x_Q).
	\end{aligned}
\]
Therefore,  {for any $\ell \in \mathbb Z$,}
\begin{equation}\label{eq- 1st eq BmL}
\begin{aligned}
\sum_{k\in \mathbb Z}&	\psi_\ell(\sqrt L) L^{s/2}[\psi_k(\sqrt L)f\phi_k(\sqrt L)g](\x)
	\\
	=&  \sum_{k\in \mathbb Z}    \sum_{Q\in \mathscr D_{k+k_0+3}}  2^{\ell s} \omega_Q \psi_{k}(\sqrt{L})f(\x_Q)   \phi_{k}(\sqrt{L})g(\x_Q)  [(2^{-\ell}\sqrt L)^s\psi_\ell(\sqrt L)]\widetilde{\Gamma}_k(\cdot,x_Q)(\x)\\
		=&  \sum_{k\in \mathbb Z}    \sum_{Q\in \mathscr D_{k+k_0+3}}  2^{(\ell-k) s} \omega_Q \widetilde{\psi}_{k}(\sqrt{L})(L^{s/2}f)(\x_Q)   \phi_{k}(\sqrt{L})g(\x_Q) \ [(2^{-\ell}\sqrt L)^s\psi_\ell(\sqrt L)]\widetilde{\Gamma}_k(\cdot,\x_Q)(\x)\\
	=&: \sum_{k \in \mathbb Z \atop k\le \ell} E_{\ell,k}+\sum_{k \in \mathbb Z \atop k>\ell} E_{\ell,k},
\end{aligned}
\end{equation}
where $\widetilde{\psi}_{k}(\sqrt L) = [2^{-k}\sqrt L]^{-s}\psi_{k}(\sqrt L)$.

\bigskip

We first estimate the term $\displaystyle \sum_{k \in \mathbb Z \atop k\le \ell} E_{\ell,k}$. By Lemmas \ref{lem-kernel-representation} and   \ref{lem-heat kernel estimate}, for $m>s/2$ and $M>n/r$,
\begin{equation}\label{eq- the use of A3}
\begin{aligned}
	\Big|[(2^{-\ell}\sqrt L)^{s }\psi_\ell(\sqrt L)]\widetilde{\Gamma}_k(\cdot,\x_Q)(\x)\Big|&= 2^{-2m\ell}\Big|[(2^{-\ell}\sqrt L)^{s-2m }\psi_\ell(\sqrt L)]L^{m}\widetilde{\Gamma}_k(\cdot,\x_Q)(\x)\Big|\\
	&=2^{-2m\ell}\Big|\int_{\RN} [(2^{-\ell}\sqrt L)^{s-2m }\psi_\ell(\sqrt L)](\x,\y)L^{m}\widetilde{\Gamma}_k(\y,\x_Q)dw(\y)\Big|\\
	&\lesi 2^{-2m\ell}  {2^{2mk}}\int_{\RN}D_{2^{-\ell},M+\N}(\x,\y)D_{2^{-k},M+\N}(\y,\x_Q)dw(\y).
\end{aligned}
\end{equation}
 {It is well known that for all $\lambda, \mu >0$,
\[\int_{\RN}D_{\lambda,M+\N}(\x,\y)D_{\mu,M+\N}(\y,\z)dw(\y)\lesi D_{\max\{\lambda , \mu\},M}(\x,\z) \quad \x,\y,\z \in \RN, \]}
Then,
\[
\int_{\RN}D_{2^{-\ell},M+\N}(\x,\y)D_{2^{-k},M+\N}(\y,\x_Q)dw(\y)\lesi D_{2^{-k},M}(\x,\x_Q),\]
and hence,
\[
	\Big|[(2^{-\ell}\sqrt L)^{s }\psi_\ell(\sqrt L)]\widetilde{\Gamma}_k(\cdot,\x_Q)(\x)\Big|\lesi 2^{-2m {(\ell-k)}}D_{2^{-k},M}(\x,\x_Q).
\]
This, together with Lemma \ref{lem1- thm2 atom Besov} and the fact $\omega_Q\simeq w(Q)$, implies
\[
\begin{aligned}
	|E_{\ell,k}(\x)|&\lesi     \sum_{Q\in \mathscr  D_{k+k_0+3}}2^{(\ell-k)(s-2m)} w(Q)  |\widetilde{\psi}_{k}(\sqrt{L})(L^{s/2}f)(\x_Q)| 	 |\phi_{k}(\sqrt{L})g(x_Q)|D_{2^{-k},M}(\x, \x_Q)\\
	& \lesi 2^{(\ell-k)(s-2m)} \mathcal M^{\mathcal O}_{r}\Big(\sum_{Q\in \mathscr D_{k+k_0+3}}|\widetilde{\psi}_{k}(\sqrt{L})(L^{s/2}f)(\x_Q)| 	 |\phi_{k}(\sqrt{L})g(x_Q)|1_{Q} \Big)(\x)  
\end{aligned}
\]
It can be verified that for $\lambda> 2\times\max\{\N/q, \N/p, \N\}$, 
\begin{equation}\label{eq - overlinePhi estimate}
\sum_{Q\in \mathscr D_{k+k_0+3}}|\widetilde{\psi}_{k}(\sqrt{L})(L^{s/2}f)(\x_Q)| 	 |\phi_{k}(\sqrt{L})g(x_Q)|1_{Q} \lesi \widetilde{\psi}_{k,\lambda}^*(\sqrt{L})(L^{s/2}f){\phi}_{k,\lambda}^*(\sqrt{L})g,
\end{equation}
where
\begin{equation*}
	\widetilde{\psi}_{k,\lambda}^*(\sqrt{L})f(\x)=\sup_{\y\in \RN}\f{|\widetilde{\psi}_{k}(\sqrt{L})f(\y)|}{(1+2^kd(\x,\y))^\lambda} \ \ \ \text{and} \ \ \ {\phi}^*_{k,\lambda}(\sqrt L)g(\x)=\sup_{\y\in \RN}\f{|{\phi}^*_{k}(\sqrt{L})g(\y)|}{(1+2^kd(\x,\y))^\lambda}. 
\end{equation*}
Hence, we further imply
\begin{equation*}
\begin{aligned}
	|E_{\ell,k}(\x)|&\lesi   2^{(k-\ell)(2m-s) }  \mathcal M_{r}^{\mathcal O}\Big( \widetilde{\psi}^*_{k,\lambda}(\sqrt{L})(L^{s/2}f){\phi}^*_{k,\lambda}(\sqrt{L})g\Big)(\x).
\end{aligned}
\end{equation*}
Consequently,
\begin{equation}\label{eq-k < ell}
\sum_{k \in \mathbb Z \atop k\le \ell} E_{\ell,k}\lesi \sum_{k \in \mathbb Z \atop k\le \ell}  2^{(k-\ell)(2m-s) }  \mathcal M_{r}^{\mathcal O}\Big( |\widetilde{\psi}^*_{k,\lambda}(\sqrt{L})(L^{s/2}f){\phi}^*_{k,\lambda}(\sqrt{L})g\Big).
\end{equation}
\bigskip

For the second term $\displaystyle \sum_{k \in \mathbb Z \atop k> \ell} E_{\ell,k}$. By Lemmas \ref{lem-kernel-representation} and   \ref{lem-heat kernel estimate}, for $k>\ell$,
\begin{equation}
	\begin{aligned}
		\Big|[(2^{-\ell}\sqrt L)^{s }\psi_\ell(\sqrt L)]\widetilde{\Gamma}_k(\cdot,\x_Q)(\x)\Big|
		&= \Big|\int_{\RN} [(2^{-\ell}\sqrt L)^{s }\psi_\ell(\sqrt L)](\x,\y)\widetilde{\Gamma}_k(\y,\x_Q)dw(\y)\Big|\\
		&\lesi  D_{2^{-\ell},M}(\x,\x_Q).
	\end{aligned}
\end{equation}
Hence,
\[
\begin{aligned}
	|E_{\ell,k}(\x)|&\lesi    \sum_{Q\in \mathscr D_{k+k_0+3}}  2^{(\ell-k)s} \omega_Q  |\widetilde{\psi}_{k}(\sqrt{L})(L^{s/2}f)(\x_Q)| |\phi_{k}(\sqrt{L})g(x_Q)|D_{2^{-\ell},M}(\x,\x_Q)\\
	&\lesi    2^{(\ell-k)s}2^{\mathfrak N(k-\ell)(1/r-1)}  \mathcal M_{r}^{\mathcal O}\Big[ \sum_{Q\in \mathscr D_{k+k_0+3}}|\widetilde{\psi}_{k}(\sqrt{L})(L^{s/2}f)(\x_Q)| 	 |\phi_{k}(\sqrt{L})g(x_Q)|1_{Q} \Big](\x),
\end{aligned}
\]
where in the last inequality we used Lemma \ref{lem1- thm2 atom Besov}.

Similarly to \eqref{eq-k < ell}, for $\lambda> 2\times\max\{\N/q, \N/p, n\} $, we have
\begin{equation}\label{eq-k > ell}
	\sum_{k \in \mathbb Z \atop k> \ell} E_{\ell,k}\lesi \sum_{k \in \mathbb Z \atop k> \ell}  2^{-(k-\ell)[s-\mathfrak N(1/r-1)] }  \mathcal M_{r}^{\mathcal O}\Big( \widetilde{\psi}^*_{k,\lambda}(\sqrt{L})(L^{s/2}f){\phi}^*_{k,\lambda}(\sqrt{L})g\Big).
\end{equation}

From \eqref{eq- 1st eq BmL}, \eqref{eq-k < ell} and \eqref{eq-k > ell}, 
\[
\begin{aligned}
	\Big|\psi_\ell(\sqrt L) L^{s/2}&\big[\sum_{k\in \mathbb Z}   \psi_k(\sqrt L)f\phi_k(\sqrt L)g\big](\x)\Big|\\
	 &\lesi  \sum_{k\in \mathbb Z\atop k<\ell}2^{k-\ell}  \mathcal M^{\mathcal O}_{r}\Big( \widetilde{\psi}^*_{k,\lambda}(\sqrt{L})(L^{s/2}f){\phi}^*_{k,\lambda}(\sqrt{L})g\Big)(\x)\\
	 &\  \ \ +  \sum_{k\in \mathbb Z\atop k\ge \ell} 2^{-(k-\ell)[s-\mathfrak N(1/r-1)]} \mathcal M^{\mathcal O}_{r}\Big( \widetilde{\psi}^*_{k,\lambda}(\sqrt{L})(L^{s/2}f){\phi}^*_{k,\lambda}(\sqrt{L})g\Big)(\x).
\end{aligned}
\]
Since $s>\mathfrak N(1/r-1)$, it follows that 
\[
\begin{aligned}
		\Big(\sum_{\ell\in\mathbb Z}|\psi_\ell(\sqrt L) L^{s/2}\big[\sum_{k\in \mathbb Z} &  \psi_k(\sqrt L)f\phi_k(\sqrt L)g\big](\x)|^q\Big)^{1/q}\\
		&\lesi  \Big(\sum_{k\in \mathbb Z}   \Big| \mathcal M^{\mathcal O}_{r}\Big( |\widetilde{\psi}^*_{k,\lambda}(\sqrt{L})(L^{s/2}f){\phi}^*_{k,\lambda}(\sqrt{L})g\Big)(\x) \Big|^q\Big)^{1/q}.
\end{aligned}
\]
By the maximal inequality \eqref{FSIn} and the H\"older inequality, we further imply
\[
\begin{aligned}
\Big\|	\Big(\sum_{\ell\in\mathbb Z}|\psi_\ell&(\sqrt L) L^{s/2}\big[\sum_{k\in \mathbb Z}   \psi_k(\sqrt L)f\phi_k(\sqrt L)g\big](\x)|^q\Big)^{1/q}\Big\|_{L^p(dw)}\\
	&\lesi  \Big\|\Big(\sum_{k\in \mathbb Z}   \Big[|\widetilde{\psi}^*_{k,\lambda}(\sqrt{L})(L^{s/2}f){\phi}^*_{k,\lambda}(\sqrt{L})g\Big]^q\Big)^{1/q}\Big\|_{L^p(dw)}\\
	&\lesi  \Big\|\Big(\sum_{k\in \mathbb Z}    \Big( |\widetilde{\psi}^*_{k,\lambda}(\sqrt{L})(L^{s/2}f)|^q\Big)^{1/q}\Big\|_{L^{p_1}(dw)}\Big\| \sup_{j\in \mathbb Z}{\phi}^*_{j,\lambda}(\sqrt{L})g\Big\|_{L^{p_2}(dw)}.
\end{aligned}
\]
At this stage, applying Propositions \ref{prop- Ls on TL B spaces}, \ref{prop2-thm1} and Lemma \ref{lem-Hardy} we arrive at
\[
\begin{aligned}
	 \Big\|\sum_{k\in \mathbb Z}   \psi_k(\sqrt L)f\phi_k(\sqrt L)g\Big\|_{\dot{F}^{s,L}_{p,q}(dw)}&\simeq \|L^{s/2}\big[\sum_{k\in \mathbb Z}   \psi_k(\sqrt L)f\phi_k(\sqrt L)g\big]\|_{\dot{F}^{0,L}_{p,q}(dw)}\\
	&\lesi \|L^{s/2}f\|_{\dot{F}^{0,L}_{p_1,q}(dw)}\|g\|_{H^{p_2}_L(dw)}\\
	&\simeq \|f\|_{\dot{F}^{s,L}_{p_1,q}(dw)}\|g\|_{H^{p_2}_L(dw)}.
\end{aligned}
\]
Similarly,
\[
\begin{aligned}
	\Big\|\sum_{k\in \mathbb Z}   \phi_{k-1}(\sqrt L)f\psi_k(\sqrt L)g\Big\|_{\dot{F}^{s,L}_{p,q}(dw)}
	&\lesi \|f\|_{H^{p_3}_L(dw)}\|g\|_{\dot{F}^{s,L}_{p_4,q}(dw)}.
\end{aligned}
\]
This completes our proof.
\end{proof}

{\bf Acknowledgements.} T. A. Bui was supported by the research grant ARC DP260101083 from the Australian Research Council, X. Han was supported by the research grants   2025KY61 and 2025AHGXZK40203 from Hefei Institute of Technology,
 and S. Mukherjee was supported by Institute Postdoctoral Fellowship
from IIT Bombay.

\end{document}